\documentclass{amsart}
\usepackage{amsfonts,amsthm,amsmath}
\usepackage{amssymb}
\usepackage{amscd}
\usepackage[utf8]{inputenc}
\usepackage[a4paper]{geometry}
\usepackage{enumerate}
\usepackage[usenames,dvipsnames]{color}
\usepackage{tikz}
\usepackage{soul}
\usepackage{array}
\usepackage{array,tabularx}

\numberwithin{equation}{section}
\numberwithin{equation}{subsection}

\theoremstyle{plain}

\newtheorem{theorem}[equation]{Theorem}
\newtheorem{lemma}[equation]{Lemma}
\newtheorem{proposition}[equation]{Proposition}

\newtheorem{corollary}[equation]{Corollary}

\theoremstyle{definition}
\newtheorem{example}[equation]{Example}
\newtheorem{remark}[equation]{Remark}

\newtheorem{definition}[equation]{Definition}

\newtheorem{bekezdes}[equation]{}

\newcommand{\bC}{{\mathbb C}}
\newcommand{\bQ}{{\mathbb Q}}
\newcommand{\bL}{{\mathbb L}}

\newcommand{\bP}{{\mathbb P}}

\newcommand{\cV}{{\mathcal V}}
\newcommand{\cA}{{\mathcal A}}

\newcommand{\cE}{{\mathcal E}}

\newcommand{\cS}{{\mathcal S}}\newcommand{\calS}{{\mathcal S}}
\newcommand{\cP}{{\mathcal P}}

\newcommand{\cO}{{\mathcal O}}\newcommand{\calO}{{\mathcal O}}
\newcommand{\cF}{{\mathcal F}}
\newcommand{\cH}{{\mathcal H}}
\newcommand{\cL}{{\mathcal L}}

\newcommand{\cZ}{\mathcal{Z}}

\newcommand{\tX}{\widetilde{X}}

\newcommand{\C}{{\calc}}

\def\blfootnote{\xdef\@thefnmark{}\@footnotetext}

\newcommand{\bt}{{\bf t}}
\newcommand{\ba}{{\bf a}}
\newcommand{\bx}{{\bf x}}

\newcommand{\setQ}{\mathbb{Q}}

\newcommand{\setZ}{\mathbb{Z}}

\newcommand{\bI}{{J(l',I)}}

\newcommand{\hh}{\mathfrak{h}}
\newcommand{\pp}{\mathfrak{p}}

\newcommand{\calv}{\mathcal{V}}

\newcommand{\Z}{\mathbb{Z}}
\newcommand{\Q}{\mathbb{Q}}

\begingroup\makeatletter\ifx\SetFigFontNFSS\undefined%
\gdef\SetFigFontNFSS#1#2#3#4#5{%
  \reset@font\fontsize{#1}{#2pt}%
  \fontfamily{#3}\fontseries{#4}\fontshape{#5}%
  \selectfont}%
\fi\endgroup%

\newcommand{\calo}{{\mathcal O}}

\def\C{\mathbb C}
\def\Q{\mathbb Q}

\def\bL{\mathbb L}
\def\Z{\mathbb Z}

\newcommand{\cale}{{\mathcal E}}

\author{J\'anos Nagy}
\address{Central European University, Dept. of Mathematics,  Budapest, Hungary}
\email{nagy\textunderscore janos@phd.ceu.edu}

\author{Andr\'as N\'emethi}
\thanks{The author is partially supported by NKFIH Grant  112735}
\address{Alfr\'ed R\'enyi Institute of Mathematics,
Hungarian Academy of Sciences,
Re\'altanoda utca 13-15, H-1053, Budapest, Hungary \newline
 \hspace*{4mm} ELTE - University of Budapest, Dept. of Geometry, Budapest, Hungary \newline \hspace*{4mm}
BCAM - Basque Center for Applied Math.,
Mazarredo, 14 E48009 Bilbao, Basque Country – Spain}
\email{nemethi.andras@renyi.mta.hu }

\title{Motivic Poincar\'e series of cusp surface singularities}

\begin{document}

\keywords{normal surface singularities, links of singularities,
plumbing graphs, rational homology spheres, cusp singularities, divisorial filtration,
Poincar\'e series, zeta functions, motivic Poincar\'e series}

\subjclass[2010]{Primary. 32S05, 32S25, 32S50, 57M27
Secondary. 14Bxx, 14J80}

\begin{abstract}
We target multivariable series associated with resolutions of complex analytic
normal surface singularities.
In general,  the equivariant multivariable analytical and topological Poincar\'e  series are well--defined and have good properties only if the link is a rational homology sphere. We wish to create a model when this assumption is not valid: we analyse the case of cusps. For such germs we define even the motivic versions of these two series, we prove that they are equal, and we provide explicit combinatorial expression for them. This is done via a motivic multivariable series associated with
the space of effective Cartier divisors of the reduced exceptional curve.
\end{abstract}

\maketitle

\linespread{1.2}



\pagestyle{myheadings} \markboth{{\normalsize A. N\'emethi}} {{\normalsize
Poincar\'e series of cusp singularities }}


\section{Introduction}

\subsection{}
Let us consider a complex analytic normal surface singularity and fix one of its
resolutions and the divisorial filtration of the local algebra associated with the
irreducible exceptional divisors.
The multivariable Poincar\'e series $\cP_0(\bt)$ associated with this
filtration  is one of the strongest analytic invariants of the germ.
For definition, particular examples  and several properties see e.g.
\cite{CHR,CDG,CDGEq,Five,Gradedroots,CDGb,NJEMS,LineBundles}.

In several special cases (even if the germ is not taut, but the analytic structure
is `nice'), $\cP_0(\bt)$ can be recovered  from the topology of the link. These cases include e.g.
the rational or minimally elliptic singularities (with rational homology sphere link)
\cite{CDGb}. The most general case when such a characterization was established is the family
of splice quotient singularities  \cite{LineBundles}.
(For such germs the analytic type is defined canonically from the graph
\cite{NWnew,NWuj2}.)

However, in all the cases when such a characterization was established the link of the corresponding germs are rational homology spheres, $\bQ HS^3$,  (that is, the dual graph is a tree of
${\mathbb P}^1$'s). In this note we wish to step over this obstruction by analysing the case
of cusps, when the dual graph consists of a loop. (The rationality of the irreducible
exceptional divisors, due to the complexity of the moduli space of algebraic curves,
presumably cannot be dropped in such topological comparison without some other essential
analytic assumption.)

 In fact, there exists even a more general series, the equivariant multivariable
Poincar\'e series $\cP(\bt)$
associated with the divisorial filtration of the local algebra of
the universal abelian  covering of the germ.
It behaves in a more uniform and conceptual way in several geometric construction, e.g.
in the context of abelian coverings and associated bundles.
However, its definition is also obstructed:
it is defined in terms of the universal abelian covering, which is well--defined only when the link is a rational homology sphere. Again,
we will step over this obstruction as well in the case of cusps.

For definition and certain properties of cusp surface singularities see \cite{Laufer77,Five,Ninv,Wagreich}.

\bekezdes When the link is a rational homology sphere there  is a concrete `topological
candidate' for $\cP(\bt)$, denoted by $\cZ(\bt)$, a multivariable series defined from
the combinatorics of the resolution graph (well--defined in any case,
even if $\cP(\bt)\not=\cZ(\bt)$). In the cases mentioned above (rational, splice quotient) in fact one has $\cP(\bt)=\cZ(\bt)$. This series $\cZ(\bt)$ has several other pleasant properties, it has deep connections with several topological invariants of the link (e.g. with the Seiberg--Witten invariant, lattice cohomology, etc) \cite{NJEMS,NICM,NLinSp}. In this way it created a bridge between analytic and topological invariants.

However, usually (e.g. for a non-rational or non--elliptic topological types)
 for a `non--nice' analytic type (e.g., for a generic one)  $\cP(\bt)\not=\cZ(\bt)$).
In such cases the challenge is to find new topological candidate series, which
might recover $\cP(\bt)$ for several `bad'  (but interesting)  families of  analytic types.
Hence different geometric  realizability (equivalent definitions)
even of `classical' $\cZ(\bt)$ is crucial: they might
suggets/impose different generalizations/versions, which might fit with different analytic structures. (This partly motivated that in the body of the article we list several parallel realizabilities.) 

\subsection{} This article has several goals. Firstly, we wish to provide a model
for the definition of $\cP(\bt)$
(in terms of analytic data)
when the link is not a rational homology sphere (hence the universal abelian
covering does not exists).
We will do this for the cusp normal surface singularities.
(Maybe here is appropriate to eliminate a possible confusion from the start:
Though cusp singularities are taut \cite{Laufer73},
hence there is no analytic moduli of their analytic structures,
the construction of $\cP(\bt)$ involves --- must involve --- certain line bundles associated with fixed Chern classes; since the corresponding Picard groups are non--trivial, these choices have an analytic moduli.)

Second, we wish to provide
 a good topological candidate $\cZ(\bt)$ for $\cP(\bt)$.
(Note that the `old' definition of $\cZ(\bt)$, though well--defined,
 for cusps gives the meaningless constant series   1, which definitely should be
 modified.)

 In order to provide these two definitions
  we rely on two collections of linear subspace arrangements,
 one of them defined analytically, the other one topologically.
 Both are indexed by the possible first Chern classes of the resolution.
 Both series in the new situation are defined via these arrangements
 by considering the topological Euler characteristic of the projectivised arrangement complements.
  For details regarding these arrangements  see  \cite{NICM,NLinSp}.
  (Since the cusp singularities are minimally elliptic, in the introductory part
  we emphasize more the known facts regarding
  the rational and minimally elliptic cases, the second one under the $\bQ HS^3$--link assumption. This serves as a good comparison for  the newly established methods and formulae.)

 Once
the definitions are settled, we prove that in the case of cusps one has $\cP(\bt)=\cZ(\bt)$ indeed. Hence, $\cZ(\bt)$ is a good topological candidate for
$\cP(\bt)$ even if the graph has 1--cycles.

In fact, the new definition  (based on the complements of subspace arrangements)
 allows us to extend both series to their
motivic versions (with coefficient in the Grothendieck group), denoted by
$\cP(\bL,\bt)$ and $\cZ(\bL,\bt)$. For these extended versions we also establish the
identity $\cP(\bL,\bt)=\cZ(\bL,\bt)$.

Surprisingly, this new extensions direct us to an unexpected new territory.
It turns out that these objects can be related with a multivariable motivic series associated with the effective Cartier divisors (supported by the reduced
exceptional curve and indexed by the possible Chern classes), denoted by
$\cZ^{{\rm ECa}}(\bL,\bt)$.
(Maybe is worth to mention that for minimal resolution of cusps the reduced exceptional divisor equals the
Artin fundamental cycle, the minial elliptic cyle  and the anticanonical cycle as well
\cite{Laufer77}.)
This connection was suggested and imposed by the
recent manuscripts of the authors regarding the Abel maps associated with a resolution of a normal surface singularity \cite{NNI,NNII,NNIII,NNIIIb}. The point is that the series
$\cZ^{{\rm ECa}}(\bL,\bt)$ has a very natural form in terms of the graph (for notations see  \ref{ss:DEFNOT}, $\cV$ and $\cE$ are the set of vertices and edges, while $E_v^*$ stay for
 the dual base elements):
$$\cZ^{{\rm ECa}}(\bL,\bt)=\frac{
\prod_{(u,v)\in\cE}\ (1-\bt^{E_u^*}-\bt^{E_v^*}+\bL \bt^{E_u^*+E_v^*})}
{\prod_{v\in\cV} \ (1-\bt^{E^*_v})(1-\bL\bt^{E^*_v})}.$$
(This is valid for any singularity --- not necessarily for a cusp ---, whenever all the
irreducible exceptional curves are rational.)
Then, in the case of cusp,  the following phenomenon happens. For relevant Chern classes $l'$, the motivic information of
any fiber of the Abel map ${\rm ECa}^{-l'}(E)\to {\rm Pic}^{-l'}(E)=\bC^*$ can be related with $\cP(\bt)$. It turns out that
 for a cusp and its minimal resolution
$$\cZ(\bL,\bt)=\cP(\bL,\bt)= 1+\frac{\cZ^{{\rm ECa}}(\bL,\bt)-1}{\bL-1}.$$
Furthermore (when the exceptional curve has at least two components) then
\begin{equation*}\begin{split}
\cP(\bt)=\cZ(\bt)&=\Big(1+\frac{\cZ^{{\rm ECa}}(\bL,\bt)-1}{\bL-1}\Big)\Big|_{\bL=1}\\
&=1+\sum_{v\in\calv} \frac{\bt^{E^*_v}}{1-\bt^{E^*_v}}+\sum _{(u,v)\in\cale}
\frac{\bt^{E^*_v}}{1-\bt^{E^*_v}}\cdot \frac{\bt^{E^*_u}}{1-\bt^{E^*_u}}.
\end{split}\end{equation*}
For the remaining case see Theorem \ref{bek:formulacusp}.

\section{Preliminaries regarding normal surface singularities}\label{s:prel}

\subsection{Definitions, notations} \label{ss:DEFNOT}
Let \((X,o)\) be a complex normal surface singularity.
 Let \(\pi:\widetilde{X}\to X\) be
a good resolution with dual graph  \(\Gamma\) whose vertices are
denoted by $\cV$ and edges by $\cE$.
Set $E:=\pi^{-1}(o)$. Let $M$ be the link of $(X,o)$.

Set \(L := H_2 ( \widetilde{X},\setZ )\). It is freely
generated by the classes of the irreducible exceptional curves.
If  \(L'\) denotes
\(H^2( \widetilde{X}, \setZ )\), then the intersection form
\((\,,\,)\) on \(L\) provides an embedding \(L \hookrightarrow
L'\) with factor the torsion part of the first homology group \(H\) of the link.
(In fact, $L'$ is the dual lattice of $(L,(\,,\,))$, it can also be
identified with $H_2(\widetilde{X}, \partial\widetilde{X},\Z)$.)
Moreover,  $(\,,\,)$ extends to
$L'$. $L'$ is freely generated by the duals \(E_v^*\), where
$ ( E_v^*, E_w) =  -1 $ for $v = w$ and
$=0$ else.

Effective classes \(l=\sum r_vE_v\in L'\) with all
\(r_v\in\setQ_{\geq 0}\) are denoted by \(L'_{\geq 0}\) and
$L_{\geq 0}:=L'_{\geq 0}\cap L$. Denote by $\cS'$ the (Lipman's)
anti-nef cone $\{l'\in L'\,:\, (l',E_v)\leq 0 \ \mbox{for all
$v$}\}$.  It is generated over $\setZ_{\geq 0}$ by the
base-elements $E_v^*$. Since   all the entries  of $E_v^*$
are strict positive, $\cS'$ is a sub-cone of $L_{\geq 0}'$,
and for any fixed $a\in L'$ the set
$\{l'\in \cS'\,:\, l'\ngeq a\}$ is finite.
Set \({\mathfrak C}:=\{\sum l'_vE_v\in L', \ 0\leq l_v' < 1\}\).
For any $l'\in L'$ write its class in $H$ by $[l']$.
Denote by $\theta:H\to \widehat{H}$ the isomorphism $[l']\mapsto
e^{2\pi i(l',\cdot)}$ of $H$ with its Pontrjagin dual
$\widehat{H}$.

Let  $K\in L'$ be the {\em canonical
class} satisfying $(K+E_v,E_v)=-2+2g_v$ for all $v\in\cV$, where $g_v$ is the genus of $E_v$. Set $\chi(l')=
-(l',l'+K)/2$. By Riemann-Roch theorem $\chi(l)=\chi(\calo_l)$ for $l\in L_{>0}$.

{\it In this preliminary section
 we will assume that $M$ is a rational homology sphere}.
(This is exactly that assumption what we wish to drop later.)
This happens if and only if \(\Gamma\) is a tree and $g_v=0$ for all $v$.

In subsections \ref{FM} and \ref{SM} we list certain analytic invariants,
then we continue with  the topological ones.
For more details regarding this part see e.g.
\cite{CHR,CDG,CDGEq,Five,Gradedroots,CDGb,NJEMS,LineBundles}.
These results are present in the literature (though slightly scattered),
 nevertheless here we collect the relevant ones.  The reason is  that this list of
constructions and results will serve as prototypes for further generalizations,
in some cases they already suggest the need of modifications as well.

\subsection{Natural line bundles}\label{FM}
Natural  line bundles are provided by the splitting of the
cohomological exponential exact sequence
 \cite[\S 4.2]{Gradedroots}:
 $$ 0 \to H^1 ( \widetilde{X}, \mathcal{O}_{\widetilde{X}} )\to
  {\rm Pic}(\widetilde{X}) \stackrel{c_1}{\longrightarrow} L' \to 0.$$
The first Chern class \(c_1\)  has an obvious
section on the subgroup \(L\), namely $l\mapsto \calo_{\widetilde{X}}(l) $.  This section has a unique extension
\(\mathcal{O}(\cdot )\) to \(L'\). We call a line bundle \emph{natural} if
it is in the image of this section.
By this definition, a line bundle $\cL$ is natural if and only if there exists a positive integer $n$ such that $\cL^{\otimes n}$ has the form $\calo_{\widetilde{X}}(l)$ for some $l\in L$.

One can recover the natural line  bundles via coverings as follows.
Let \(c:(X_a,o)\to (X,o)\) be the universal abelian covering of
\((X,o)\). (Note that the existence of the universal abelian covering is guaranteed
by the fact that the link is a rational homology sphere.
Indeed $c$ is a finite regular covering over $X\setminus \{o\}$
and this regular covering has a unique
non-regular  extension to the level of germs of normal surface singularities. The regular covering
is associated with the representation $\pi_1(M)\to H_1(M,\Z)= {\rm Tors}(H_1(M,\Z))=H$.) Furthermore, let
  \(\pi_a:\widetilde{X}_a\to X_a\) the normalized pullback of
\(\pi\) by \(c\), and
\(\widetilde{c}:\widetilde{X}_a\to\widetilde{X}\) the morphism which
covers \(c\). Then the action of \(H\) on \((X_a,o)\) lifts to an action on
\(\widetilde{X}_a\) and one has an $H$--eigensheaf  decomposition
(\cite[4.2.9]{Gradedroots} or \cite[(3.5)]{Opg}):
\begin{equation}\label{eq:01}
\widetilde{c}_*\cO_{\widetilde{X}_a}=\bigoplus _{l'\in {\mathfrak C}}
\cO_{\widetilde{X}}(-l') \ \ \ \
 (\mbox{$\cO_{\widetilde{X}}(-l')$ being the $\theta([l'])$-eigenspace of $\widetilde{c}_*\cO_{\widetilde{X}_a}$}).
\end{equation}
Then write  any $l'\in L'$ as $l'_0+l$ with $l'_0\in \mathfrak{C}$ and $l\in L$, and  set
$\calo_{\widetilde{X}}(-l'):=\calo_{\widetilde{X}}(-l'_0)\otimes
\calo_{\widetilde{X}}(-l)$.

\subsection{Series associated with the divisorial filtration.}\label{SM}
Once a resolution
$\pi$ is fixed, $\cO_{X_a,o}$ inherits the {\em divisorial
multi-filtration} (cf. \cite[(4.1.1)]{CDGb}):
\begin{equation}\label{eq:03}
\cF(l'):=\{ f\in \cO_{X_a,o}\,|\, {\rm div}(f\circ\pi_a)\geq \widetilde{c}^*(l')\}.
\end{equation}
Let $\hh(l') $ be the dimension of the $\theta([l'])$-eigenspace
of $\cO_{X_a,o}/\cF(l')$. Then, one defines  the {\em equivariant
divisorial Hilbert series} by
\begin{equation}\label{eq:04}
\cH(\bt)=\sum _{l'\in L'}
\hh(l')t_1^{l_1'}\cdots t_s^{l_s'}=\sum_{l'\in L'}\hh(l')\bt^{l'}\in
\setZ[[L']] \ \ (l'=\sum _i l'_iE_i).
\end{equation}

Notice that the terms of the sum reflect the $H$-eigenspace
decomposition too: $\hh(l')\bt^{l'}$ contributes to
the $\theta([l'])$-eigenspace.  For example, $\sum_{l\in L}\hh(l)\bt^{l}$ corresponds
to the $H$-invariants, hence it is the {\em Hilbert series}  of
$\cO_{X,o}$ associated with the $\pi^{-1}(o)$-divisorial
multi-filtration.

The `graded version' associated with the
Hilbert series is defined (cf. \cite{CDG,CDGEq}) as
\begin{equation}\label{eq:06}
\cP(\bt)=-\cH(\bt) \cdot \prod_v(1-t_v^{-1})\in \setZ[[L']].
\end{equation}
Usually this is called the equivarient multivariable analytic Poincar\'e series of $(X,o)$.

If we write the series $\cP(\bt)$ as $\sum_{l'} \pp(l')\bt^{l'}$, then
\begin{equation} \label{eq:1}
\pp(l')=\sum_{I \subseteq \cV}
(-1)^{\left\lvert I \right\rvert + 1}
  \dim \frac{H^0(\widetilde{X}, \calo_{\widetilde{X}}(-l'))}
   {H^0(\widetilde{X}, \calo_{\widetilde{X}}(-l'-E_I  ))}
\end{equation}
and $\cP$ is supported in the cone $\calS'$.

 Although the
multiplication by $\prod_v(1-t_v^{-1})$ in $\setZ[[L']]$ is not
injective, hence apparently $\cP$ contains less information then
$\cH$, they, in fact, determine each other. Indeed,  for any $l'\in L'$ one has
\begin{equation}\label{eq:10b}
\hh(l')=\sum_{l\in L,\, l\not\geq 0} \pp(l'+l).
\end{equation}

\subsection{Linear subspace arrangements associated with the filtration}\cite{NLinSp,NICM,NBOOK}
\label{bek:PasArr}
Fix a normal surface singularity, one of its
resolutions and the filtration $\{\cF(l')\}_{l'\in L'}$.
For any $l'\in L'$,  the linear space
 $$(\cF(l')/\cF(l'+E))_{\theta([l'])}=H^0(\cO_{\tX}(-l'))/H^0(\cO_{\tX}(-l'-E))$$
naturally embeds  into
$$T(l'):=H^0(\cO_E(-l')).$$
Let its image be denoted by $A(l')$.
Furthermore,
for every $v\in \cV$, consider the linear subspace $T_v(l')$ of $T(l')$
given by
$$T_v(l'):=H^0(\cO_{E-E_v}(-l'-E_v))=\mathrm{ker}\,(\, H^0(\cO_E(-l'))\to H^0(\cO_{E_v}(-l'))\,)\subset
T(l').$$
Then the image $A_v(l')$ of $ H^0(\cO_{\tX}(-l'-E_v)/H^0(\cO_{\tX}(-l'-E))$ in $T(l')$
satisfies $A_v(l')=A(l')\cap T_v(l')$.
\begin{definition}\label{def:arr}
The (finite dimensional) arrangement of linear subspaces $\cA_{{\rm top}}(l')
=\{T_v(l')\}_v$ in $T(l')$
is called the `topological arrangement' at $l'\in L'$.
The arrangement of linear subspaces $\cA_{{\rm an}}(l')=
\{A_v(l')=T_v(l')\cap A(l')\}_v$ in $A(l')$
is called the `analytic arrangement' at $l'\in L'$.
The corresponding projectivized arrangement complements will be denoted by
$\bP (T(l')\setminus \cup_v T_v(l'))$ and $\bP (A(l')\setminus \cup_v A_v(l'))$ respectively.
\end{definition}
If $l'\not\in \cS'$ then there exists $v$ such that $(E_v,l')>0$, that is $h^0(\cO_{E_v}(-l'))=0$,
proving  that $T_v(l')=T(l')$. Hence $A_v(l')=A(l')$ too. In particular, both arrangement complements
are empty.

The inclusion--exclusion principle and $\dim V_I=\chi_{{\rm top}}(\bP V_I)$ shows the following fact.

\begin{lemma}\label{lem:Arrang}
Assume that $\{V_\alpha\}_{\alpha\in \Lambda}$
is a finite family of linear subspaces of a finite dimensional
linear space $V$. For $I\subset \Lambda $ set $V_I:=\cap_{\alpha\in I}V_\alpha$
(where $V_\emptyset=V$). Then
$$\chi_{{\rm top}} ( \bP (V\setminus\cup _\alpha V_\alpha))=\sum_{I\subset \Lambda} \, (-1)^{|I|}
\dim V_I.$$
If $\Lambda\not=\emptyset$, then this also equals
$\sum_I(-1)^{|I|+1}\mathrm{codim}(V_I\subset V)$.
\end{lemma}
In particular, this lemma and identity (\ref{eq:1}) imply the following.
\begin{corollary}\label{cor:PArr} For any $l'\in \cS'$ one has
$$\pp(l')=\chi_{{\rm top}}(\, \bP(A(l')\setminus \cup_{v}A_v(l'))\,).$$
\end{corollary}

Then for any $l'\in L'$ and $I\subset \calv$ one has:
\begin{equation}\label{lem:arr}
\dim\,A(l')=\hh (l'+E)-\hh(l'), \ \ \ \dim\, \cap_{v\in I}A_v(l')=\hh(l'+E)-\hh(l'+E_I).
\end{equation}
Thus, we can expect that the analytic arrangement is rather sensitive to the
modification of the analytic structure, and in general, does not coincide with the
topological arrangement.

In the next paragraphs we will  show that whenever the link of the singularity
is a rational homology sphere the {\em topological arrangement} $\cA_{{\rm top}}$
is indeed topological, it depends only on the combinatorics of the resolution graph.
We will need the following technical definition.

\begin{lemma}\label{DL}  \cite{CDGb}
\begin{enumerate}
\item\label{DL1}
For any  $l'\in L'$ and subset $I\subset \cV$ there exists a
unique minimal subset $\bI\subset \cV$ which contains $I$, and
has the following property:
\begin{equation}\label{star}
\mbox{there is no $v\in \cV\setminus \bI$ with
$(E_v,l'+E_\bI)>0$.}\end{equation}

\item\label{DL2} $\bI$ can
be found by the next algorithm: one constructs a sequence
$\{I_m\}_{m=0}^k$ of subsets of $\cV$, with $I_0=I$,
$I_{m+1}=I_m\cup \{v(m)\}$, where the index $v(m)$ is determined
as follows. Assume that $I_m$ is already constructed. If $I_m$
satisfies (\ref{star}) we stop and $m=k$. Otherwise, there exists
at least one $v$ with $(E_v,l'+E_{I_m})>0$. Take $v(m)$ one of
them and continue the algorithm  with $I_{m+1}$. Then
$I_k=\bI$.
\end{enumerate}
\end{lemma}

\begin{proposition}\label{prop:tarr} \cite{CDGb}
Assume that the resolution graph is a tree of rational curves. For any $l'\in L'$
and $I\subset \cV$ write $J(I):=J(l',I)$. Then the following facts hold.
\begin{enumerate}
\item[(a)] One has the following commutative diagram with exact rows

\begin{picture}(400,85)(0,-10)
\put(150,60){\makebox(0,0){$
0\ \to \ H^0(\cO_{E-E_{J(I)}}(-l'-E_{J(I)}))\ \ \to \ \ H^0(\cO_{E}(-l'))\ \
\stackrel{k}{\twoheadrightarrow} \ \  H^0(\cO_{E_{J(I)}}(-l')) \ \to \ 0$}}
\put(148,30){\makebox(0,0){$
\ 0\ \to  \ \  \ H^0(\cO_{E-E_I}(-l'-E_I))\ \ \ \ \ \
\to \ \ H^0(\cO_{E}(-l'))\ \
\to \ \  H^0(\cO_{E_I}(-l')) \hspace{1cm} $}}
\put(115,0){\makebox(0,0){$\cap_{v\in I}\, T_v(l') \ \ \ \ \ \ \ \
\hookrightarrow \ \ \ \ \ \ \
T(l')$}}
\put(240,45){\makebox(0,0){$\downarrow \,   i$}}
\put(160,50){\line(0,-1){8}}\put(162,50){\line(0,-1){8}}
\put(75,45){\makebox(0,0){$\downarrow \, j$}}
\put(63,45){\makebox(0,0){$\simeq$}}
\put(160,20){\line(0,-1){8}}\put(162,20){\line(0,-1){8}}
\put(72,20){\line(0,-1){8}}\put(74,20){\line(0,-1){8}}
\end{picture}

\noindent where $j$\, is an isomorphism
(hence $\cap_{v\in I}\, T_v(l')=\cap_{v\in J(I)}\, T_v(l')$),
 $i$\,
 is injective and $k$\, is surjective.

\item[(b)] $\dim\,\cap_{v\in J(I)}T_v(l') =
\chi(\cO_{E-E_{J(I)}}(-l'-E_{J(I)}))$\, $=\chi(l'+E)-\chi(l'+E_{J(I)})$.

\item[(c)]
In particular,  if $J(I_1)=J(I_2)$ then $\cap_{v\in I_1}T_v(l')=\cap_{v\in I_2}T_v(l')$, and
 if $J(I_1)\varsubsetneq J(I_2)$ then
$\cap_{v\in I_1}T_v(l')\varsupsetneq\cap_{v\in I_2}T_v(l')$.
Therefore, $J(I)$ is the unique maximal
subset $I_{max}\subset \cV$, such that $I\subset I_{max}$, and
$\cap_{v\in I}T_v(l')=\cap_{v\in I_{max}}T_v(l')$.

\item[(d)] Part (b) for $I=\emptyset$ reads as follows:
    $\dim\,T(l')=
    \dim\,\cap_{v\in J(\emptyset)}T_v(l')
     =\chi(l'+E)-\chi(l'+E_{J(\emptyset)})$.
     Hence, if\, $l'\in\cS'$ then   $\dim\,T(l')=-(l',E)+1$.
\item[(e)] ${\rm codim}\, (\,\cap_{v\in I}T_v(l')\hookrightarrow T(l')\,)=
\chi(l'+E_{J(I)})-\chi(l'+E_{J(\emptyset)})$.
\item[(f)] In particular, the arrangement complement is non--empty if and only if $J(\emptyset)=\emptyset$ (if and only if\, $l'\in\cS'$).

\end{enumerate}
\end{proposition}
Therefore,
if the graph is a tree of rational curves then the isotopy type of the
arrangement $\cA_{{\rm top}}$ depends only on the combinatorial data of
the graph.
Note also  that $J(l', I)$,
and  the topological linear subspace arrangement $\cA_{{\rm top}}$ too,
depend only on the $E^*$--coefficients of $l'$ and on the {\it shape of the graph $\Gamma$}, that is, on the valencies $\{\kappa_v\}_{v\in \calv}$ but not on the Euler numbers
$\{E_v^2\}_v$.

At topological Euler characteristic level one has:
\begin{corollary}\label{cor:tarreu} If the graph is a tree of rational curves and $l'\in \cS'$ then
$$\chi_{{\rm top}}(\, \bP(T(l')\setminus \cup_{v}T_v(l'))\,)=
\sum_{I\subset \cV}(-1)^{|I|+1}\chi(l'+E_{J(l',I)}).$$
\end{corollary}
\begin{proof}
Use Lemma \ref{lem:Arrang} and Proposition \ref{prop:tarr}(b).
\end{proof}

\begin{example} \label{MTH}
Using special vanishing theorems and computation sequences of rational and elliptic
singularities (cf. \cite{Ninv,Five}) one can
prove the following results as well  (see e.g. \cite{NBOOK,NLinSp}).

(I) Assume the following situations:
\begin{enumerate}
\item[(a)]\label{rat}
either $(X,o)$ is rational,  $\pi$ is arbitrary resolution, and $l'\in \cS'$ is arbitrary,
\item[(b)]\label{me}
or $(X,o)$ is minimally elliptic singularity with $H^1(\tX,\Z)=0$,  $\pi$ is a  resolution
such that the support of the
elliptic cycle equals $E$, and we also assume that for the fixed  $l'\in \cS'$
there exists a computation sequence $\{x_i\}_i $ for the fundamental cycle $Z_{min}$ (in the sense of Laufer \cite{Laufer72}),
which contains $E$ as one of its terms,
and it jumps (that is, $(x_i,E_1)=2$) at some $E_1$ with $(E_1,l')<0$.
\end{enumerate}
Then  the topological and analytic arrangements at $l'$ agree,
$\cA_{{\rm top}}(l')=\cA_{{\rm an}}(l')$. 

(II)  For minimally elliptic singularities it can happen that $\cA_{{\rm top}}(l')\not=\cA_{{\rm an}}(l')$,
even for the minimal resolution. E.g., in the case of the minimal good resolution of
$\{x^2+y^3+z^7=0\}$, or in the case of minimal resolution of $\{x^2+y^3+z^{11}=0\}$ (which is good),
for $l=Z_{min}$ one has  $\dim( T(Z_{min}))=2$ and  $\dim( A(Z_{min}))=1$.

(III) For any $l'\in\calS'$ one has the exact sequence
$$0\to A(l')\to T(l') \to H^1(\calO_{\tX}(-l'-E))\to H^1(\calO_{\tX}(-l'))$$
Hence, $\cA_{{\rm an}}(l')=\cA_{{\rm top}}(l')$ whenever $H^1(\calO_{\tX}(-l'-E))=0$.
This happens e.g. if $l'=\sum _va_vE^*_v$ with $a_v\gg 0$, in which case
$H^1(\calO_{\tX}(-l'-E))=0$ by the Grauert--Riemenschneider Vanishing Theorem.
\end{example}

\subsection{The topological series $\cZ(\bt)$}\cite{CDG,CDGEq,Five,NLinSp,NICM,NBOOK}\label{mi}
 The series $\cZ(\bt)\in \setZ[[\cS']]$
is defined by
the rational function $z({\bf x})$ in variables $x_v=\bt^{E_v^*}$, or by its Taylor
expansion at the origin,  where
\begin{equation}\label{zetaelso}
z({\bf x}):=\prod_{v\in \cV}(1-x_v)^{\kappa_v-2},
 \end{equation}
 and $\kappa_v $ is the valency of the vertex $v$.
Hence it is the expansion of $\cZ(\bt)=\prod_v (1-\bt^{E_v^*})^{\kappa_v-2}$.

We start to list some other appearances of $\cZ(\bt)$.

If $\Sigma$ is a topological space, let $S^a\Sigma$ ($a\geq 0$) denote its symmetric
product $\Sigma^a/{\mathfrak S}_a$. For $a=0$, by convention, $S^0\Sigma$ is a point. Then, by
Macdonald formula \cite{MD},
\begin{equation}\label{eq:MD}
\sum_{a\geq 0} \chi_{{\rm top}}(S^a\Sigma)\,x^a=(1-x)^{-\chi(\Sigma)}.
\end{equation}
Since  $E_v\simeq \bP^1$, and $E_v^\circ$ is the
regular part of $E_v$, then $\chi_{{\rm top}}(E_v^\circ)=2-\kappa_v$.

\vspace{2mm}

\noindent
{\bf The first formula  of $Z(\bt)$  \cite{CDG,CDGEq}.}
With the notation  $\bx^\ba  = x_1^{a_1}\cdots x_s^{a_s}$,
\begin{equation}\label{eq:Zchi}\begin{split}
z(\bx)&=\prod_v \ \sum_{a_v\geq 0}\chi_{{\rm top}}(S^{a_v}E^\circ _v)\,x^{a_v}_v=
 \sum_{\ba\geq 0}\ \prod_v \chi_{{\rm top}}(S^{a_v}E^\circ _v) \bx^{\ba}\\
&=\sum_{\ba\geq 0}\Big( \prod_v (-1)^{a_v}\binom{\kappa_v-2}{a_v}\Big) \ \bx^{\ba},\end{split}
\end{equation}
where, for {\it any} integer $b$, $\binom{b}{0}=1$ and $\binom{b}{a}=b(b-1)\cdots (b-a+1)/a!$
\ for $a\in \Z_{>0}$ as usual.

The next interpretation of $\cZ(\bt)$ is in terms of $J(l',I)$,  cf.  \ref{DL}.

\vspace{2mm}

\noindent {\bf The second formula of $\cZ(\bt)$ \cite{CDGb}.}
\begin{equation}\label{EQ:100}
\cZ(\bt)=\sum_{l'\in\cS'}\ \sum_{I\subset \cV}\
(-1)^{|I|+1}\chi(l'+E_\bI){\bt}^{l'}.
\end{equation}

\begin{remark}\label{rem:PcomparedZ} The above formula can be compared with
\begin{equation*}
\cP(\bt)=\sum_{l'\in\cS'}\ \sum_{I\subset \cV}\
(-1)^{|I|+1}\Big( \ \chi(l'+E_{J(l',I)})-h^1(\cO_{\tX}(-l'-E_{J(l',I)}))\Big) \bt^{l'}.
\end{equation*}
This combined with (\ref{EQ:100}) gives
\begin{equation*}
\cZ(\bt)-\cP(\bt)=\sum_{l'\in\cS'}\ \sum_{I\subset \cV}\
(-1)^{|I|+1}\ h^1(\cO_{\tX}(-l'-E_{J(l',I)}))\ \bt^{l'}.
\end{equation*}
\end{remark}

\vspace{2mm}

\noindent
 {\bf The third formula  of $\cZ(\bt)$.}
\begin{equation}\label{cor:tarrZ}
\cZ(\bt)=\sum_{l'\in \cS'}\, \chi_{{\rm top}}(\, \bP(T(l')\setminus \cup_{v}T_v(l'))\,)
\cdot \bt^{l'}.\end{equation}
This follows from the combination of  Corollary \ref{cor:tarreu} and (\ref{EQ:100}).
Then, by  \ref{MTH} and Corollaries  \ref{cor:PArr} and \ref{cor:tarrZ} one also has:

\begin{corollary}\label{MTH2}  $\cP(\bt)=\cZ(\bt)$ in the following cases:
\begin{enumerate}
\item[(a)]\label{rat222}
$(X,o)$ is rational, and  $\pi$ is arbitrary resolution,
\item[(b)]\label{me222}
or $(X,o)$ is minimally elliptic singularity, and it satisfies the assumptions of  \ref{MTH}(I).
\end{enumerate}\end{corollary}
In fact, under the condition of Corollary \ref{MTH2}, in \cite[p. 280-281]{CDGb}
it is proved that
 \begin{equation}\label{eq:cusp**}\cP(\bt)=\bt^0+\sum_{l'\in \calS'\setminus \{0\}}
 \sum_I (-1)^{|I|+1} \chi(l'+E_{J(l',I)})\cdot \bt^{l'}.\end{equation}

\begin{remark}\label{rem:MTH}
Part {\it (b)} can be improved by adding some additional cases when $\cA_{{\rm an}}(l')\not=
\cA_{{\rm top}}(l')$, but the Euler characteristics of the two arrangement complements  agree.
 E.g.,  if $(X,o)$ is minimally elliptic singularity whose
minimal resolution is good, and if $\pi$ is this minimal resolution,  then  $\cP(\bt)=\cZ(\bt)$.
In general,  $\cP(\bt)=\cZ(\bt)$ if and only if $(X,o)$ is a splice quotient singularity \cite{LineBundles,NICM}.
\end{remark}

\subsection{$\mathbf{\bP(T(l')\setminus \cup_{v}T_v(l'))}$ as a
space of effective  Cartier divisors}\label{bek:Eca} \cite{NNI,NBOOK}
For any cycle $Z\in L$, $Z\geq E$, let ${\rm ECa}(Z)$ be the set of effective
Cartier divisors on $Z$.  Their supports are zero--dimensional in $E$. Taking the
class of a Cartier  divisor provides the {\it Abel map} ${\rm ECa}(Z)\to {\rm Pic}(Z)$.
Let ${\rm ECa}^{l'}(Z)$ be the subset of ${\rm ECa}(Z)$, which consists of divisors
whose associated line bundles
have Chern class $l'\in L'$. Set  $c^{l'}:{\rm ECa}^{l'}(Z)\to
{\rm Pic}^{l'}(Z)$  for the restricted Abel map.
Regarding the existence of ${\rm ECa}^{l'}(Z)$ as an algebraic variety
we make the following comment.
First, by a theorem of Artin \cite[3.8]{Artin69}, there exists an affine algebraic variety $Y$
and a point $y\in Y$ such that $(Y,y)$ and $(X,o)$ have isomorphic formal completions.
Then, according to Hironaka \cite{Hironaka65},
$(Y,y)$ and $(X,o)$ are analytically isomorphic. In particular, we can regard $Z$
as a projective  algebraic scheme, in which case ${\rm ECa}^{l'}(Z)$ together with the algebraic Abel map,
 as part of the general theory,
was constructed by Grothendieck \cite{Groth62}, see also  the article of
Kleiman \cite{Kleiman2013}.  In particular,
$c^{l'}:{\rm ECa}^{l'}(Z)\to {\rm Pic}^{l'}(Z) \ \mbox{ is algebraic}$.
For an explicit description and several properties see \cite{NNI,NNII,NNIII}.

From definition, ${\rm ECa}^{l'}(Z)\not=\emptyset$ if and only if there exists $\cL\in
{\rm Pic}^{l'}(Z)$, which  has a global section vanishing somewhere, but it has no
fixed components.  If this happens then  $l'\not=0$ and $H^0(\cL|_{E_v})
\not =0$ for any $v$, hence $-l'\in \cS'\setminus \{0\}$. Conversely, if
$-l'\in \cS'\setminus \{0\}$ then one constructs elements of ${\rm ECa}^{l'}(Z)$
by some generic cuts of $E$ enumerated by $l'$. Hence, it is natural
 to modify the definition, and redefine formally ${\rm ECa}^0(Z)$ as a point, the space of the `empty divisor' $\{\emptyset\}$. It is sent by the Abel map to $\calO_Z$.
Hence, finally, ${\rm ECa}^{l'}(Z)\not=\emptyset$ if and only if $-l'\in \cS'$.

In \cite{NNI} is proved that for any $-l'\in \cS'$,  ${\rm ECa}^{l'}(Z)$ is irreducible, quasiprojective, smooth
and of dimension $-(l',Z)$.

For any $\cL\in {\rm Pic}^{l'}(Z)$ define $H^0(Z, \cL)_{{\rm reg}}$, the set of regular sections (or, section without fixed components)
 by $H^0(Z,\cL)\setminus \cup_v H^0(Z-E_v,\cL(-E_v))$. Then the preimage of $\cL$ by the Abel map is $H^0(Z,\cL)_{{\rm reg}}/H^0(\calO_Z^*)$ \cite[\S 3]{Kl}.

  Next, assume that $Z=E$ and $l'\in\cS'$. In this case $h^1(\calO_E)=0$, ${\rm Pic}^{-l'}(E)$
  consists of a point, say   $ \{\calO_{E}(-l' )\}$, and
  $H^0(\calO_E^*)=\C^*$. Hence,  by the above discussion,
  \begin{equation}\label{eq:ECA}
  {\rm ECa}^{-l'}(E)=
  H^0(\calO_E(-l'))_{{\rm reg}}/\C^*= \bP(T(l')\setminus \cup_{v}T_v(l')).
  \end{equation}

\vspace{2mm}
\noindent  {\bf  The forth  formula  of $\cZ(\bt)$.}
 If the link is a rational homology sphere then
 $$\cZ(\bt)=\sum_{l'\in L'} \chi_{{\rm top}}({\rm ECa}^{-l'}(E))\cdot \bt^{l'}.$$

\subsection{The extension of $\cZ(\bt)$ to the Grothendieck ring.}\label{miGroth}
The information contained in  $\cZ(\bt)$ can be improved
if we modify  the `third formula'
$\cZ(\bt)=\sum_{l'\in\cS'} \chi_{{\rm top}} (\bP(T(l')\setminus \cup_v T_v(l'))
\bt^{l'}$. Namely, we replace
the topological Euler characteristic of $\bP(T(l')\setminus \cup_v T_v(l'))$
with the class of this space in the Grothendieck group of complex quasi--projective varieties.
\begin{equation}\label{eq:ZGroth}
\cZ(\bL,\bt)=\sum_{l'\in\cS'} \, [\bP(T(l')\setminus \cup_v T_v(l'))]\, \bt^{l'}.
\end{equation}
Let $\bL$ be the class of the 1--dimensional affine space. Then, by
inclusion--exclusion principle (as the analogue of \ref{lem:Arrang})
one has the following. If $\{V_\alpha\}_{\alpha\in\Lambda}$
is a finite family of linear subspaces of a finite dimensional linear space $V$, and for $I\subset \Lambda $ one writes $V_I:=\cap_{\alpha\in I}V_\alpha$, then
\begin{equation*}\label{eq:Groth}
[V\setminus \cup_\alpha V_\alpha]=\sum_I(-1)^{|I|} \bL^{\dim(V_I)},\ \
[\bP(V\setminus  \cup_\alpha V_\alpha)]=(\sum_I(-1)^{|I|} \bL^{\dim(V_I)})/(\bL-1).
\end{equation*}
Hence, using \ref{prop:tarr}, (\ref{eq:ZGroth}) reads as
\begin{equation}\label{eq:ZGroth2}
\begin{split}
\cZ(\bL,\bt)& =\frac{1}{\bL-1}\cdot
\sum_{l'\in\cS'} \, \sum_{I\subset \cV} (-1)^{|I|}\,
\bL^{\chi(l'+E)-\chi(l'+E_{J(l',I)})}\, \bt^{l'}\\
& =
\sum_{l'\in\cS'} \, \sum_{I\subset \cV} (-1)^{|I|}\ \cdot
\frac{ \bL^{\chi(l'+E)-\chi(l'+E_{J(l',I)})}-1}{\bL-1} \ \bt^{l'}.
\end{split}
\end{equation}
Note that $\lim_{\bL\to 1}\cZ(\bL,\bt)=\cZ(\bt)$.
The analogue of the topological/combinatorial identity (\ref{EQ:100})
is:
\begin{theorem}\label{th:Groth} \ \cite{Nagy,NICM,NLinSp}
\begin{equation}\label{eq:motZL}
\cZ(\bL,\bt)=\frac{
\prod_{(u,v)\in\cE}\ (1-\bt^{E_u^*}-\bt^{E_v^*}+\bL \bt^{E_u^*+E_v^*})}
{\prod_{v\in\cV} \ (1-\bt^{E^*_v})(1-\bL\bt^{E^*_v})}.
\end{equation}
\end{theorem}
\noindent
One defines similarly the analytic version as well:
$\cP(\bL,\bt)=\sum_{l'\in\cS'} \, [\bP(A(l')\setminus \cup_v A_v(l'))]\, \bt^{l'}$ \cite{CDGMot}. (This will be improved in the next section, cf. Example \ref{ex:recover}.

\section{The extension of the series  to cusp singularities.}\label{ss:CUSPS}
\subsection{Notations and preliminaries regarding cusps}\label{ss:PrelCusps}
Assume that $(X,o)$ is an arbitrary normal surface singularity with
$b_1(M)={\rm dim}\, H_1(M,\Q)$ not necessarily zero.
Then the long exact sequence of the pair
$(\widetilde{X},\partial \widetilde{X})$ with $L=H_2(\widetilde{X},\Z)$
and $L'=H_2(\widetilde{X},\partial \widetilde{X},\Z)$ gives
$$0\to L\to L' \to H_1(M,\Z)\to H_1(\widetilde{X},\Z)=\Z^{b_1(M)}\to 0,$$
where $b_1(M)=2\sum_{v}g_v+c(\Gamma)$, $c(\Gamma)$ being
the number of independent cycles in the dual graph $\Gamma$. Hence, in this case
$H:=L'/L$ is identified with the torsion part ${\rm Tors}H_1(M,\Z)$.

The point is that in general there exists no canonical splitting of the exact sequence
$0\to H\to H_1(M,\Z)\to \Z^{b_1(M)}\to 0$. Different choices of  splittings
$H_1(M,\Z)\to H$ composed with $\pi_1(M)\to H_1(M,\Z)$ provide essentially
different representations $\pi_1(M)\to H$, hence different $H$--coverings.
In particular, via such representations the definition of the natural line bundles
(following the method of \ref{FM}) is not well--defined. In fact, in general, any other
definition of the natural line bundles fails (either by this ambiguity, or by the
fact that ${\rm Pic}^0(\widetilde{X})=H^1(\widetilde{X}, \calo_{\widetilde{X}})/
H^1(\widetilde{X},Z)$ is not torsion free).

On the other hand, all other combinatorial invariants, e.g. $\calS'$, are defined similarly, see e.g. \cite{Five,Ninv}.

In the case of cusp singularities, $g_v=0$ for all $v\in\calv$, and $b_1(M)=c(\Gamma)=1$.
Furthermore, since $p_g=1$ we also have ${\rm Pic}^0(\widetilde{X})=\C/\Z$.

In the previous section, in the definition of the series
 $\cP(\bt)$,   we used the
assumption that the link is rational homology sphere.
This  was really necessary, since the definition was based on
the existence of
the natural line bundles (defined via the universal abelian covering).
The point is that
usually the cohomological properties  of a natural line bundle and of a line bundle with the same Chern class differ, hence the identification of the natural bundles is crucial.
(On the other hand, if we wish to define only the $H$--equivariant part
of $\cP(\bt)$, namely $P_0(\bt):=\sum_{l\in L}\mathfrak{p}(l)\bt^l$,
then we do not need any covering, and it can be defined as the Poincar\'e
series of the divisorial filtration of $\calo_{X,o}$ associated with the irreducible
components of $\pi^{-1}(0)=E$. Nevertheless, in this way we loose essential part of the theory, namely all the geometry related with not integral Chern classes.)

Let us look also at the `old' definition of $\cZ(\bt)$ too, cf. \ref{mi}.
E.g., for the minimal resolution of cusps (when
$\kappa_v=2$ for all $v$),
(\ref{zetaelso}) gives $\cZ(\bt)\equiv 1$, an object which definitely carries no information.

Hence, in general, the possible extensions  of the series $\cP(\bt)$ and $\cZ(\bt)$ are seriously obstructed.

However, in this section we explain  that an  extension can be done even if the link is not rational homology sphere, at least in the case of cusps.
This might serve as a model for further generalizations (at least for the cases
when all the exceptional curves are rational).
In the case of cusps several analytic vanishing statements are present, facts which make the definition and results work.

\subsection{The series $\cP(\bt)$ for cusps}\label{ss:PCusps}
Let us assume that $(X,o)$ is a cusp singularity, and we fix its minimal
resolution $\tX$. The definition of $\cP(\bt)$ can be done in two different ways.
The first one is a `naive' one: one defines  $\cP(\bt)$ by the identity (\ref{eq:1}),
  \begin{equation}\label{eq:CUSPP}
 \pp(l')=\sum_{I\subset \cV}\, (-1)^{|I|}\, \dim \,\frac{H^0(\tX, \cO_{\tX}(-l'-E_I))}
  {H^0(\tX, \cO_{\tX}(-l'-E))}.\end{equation}
once we clarify the meaning of
 $\calO_{\tX}(-l')\in {\rm Pic}(\tX)$ for any $l'\in \calS'$.

 Before we make the choice of  $\calO_{\tX}(-l')\in {\rm Pic}(\tX)$ we make two
 remarks. Let us fix any line bundle $\cL\in {\rm Pic}(\tX)$ with $c_1(\cL)=l'$.
 Then, for any effective $l\in L$,
 from the cohomological exact sequence of $0\to \cL(-l)\to \cL\to \cL|_l\to 0$,
 we have
 $$\dim \, H^0(\cL)/H^0(\cL(-l))-\chi(l)-(l,l')+h^1(\cL(-l))-h^1(\cL)=0.$$
 On the other hand, by a Laufer type algorithm,
 for any $\cL\in{\rm Pic}(\widetilde{X})$, there exists $l\in L_{\geq 0}$ such that
 $c_1(\cL\otimes \calo_{\widetilde{X}}(-l))\in-\calS'$ and
 $h^1(\cL)-h^1(\cL\otimes \calo_{\widetilde{X}}(-l))$ is topological (see e.g.
 \cite[Prop. 4.3.3]{Gradedroots}. Next, for $\cL\in {\rm Pic}(\tX)$ with $c_1(\cL)
 \in -\calS'$ one has (cf. \cite{Laufer77}, \cite[p. 333] {Five}, \cite[1.7]{Rorh}, \cite{NBOOK})
 \begin{equation}\label{eq:vanishing}
 h^1(\widetilde{X},\cL)=0\ \ \mbox{unless $\cL=\calo_{\widetilde{X}}$, when
  $ h^1(\widetilde{X},\calo_{\widetilde{X}})=1$.}
\end{equation}
 In particular, in all our relevant cases in the computation of
  the $\mathfrak{p}(l')$ in
 (\ref{eq:CUSPP}) for $l'\in \calS'$, the expression from the right hand side
 is independent on the choice of  $\calO_{\tX}(-l')\in {\rm Pic}(\tX)$,
 basically it depends (topologically/combinatorially) only on $l'$.

 The second definition identifies precisely the bundles  $\calO_{\tX}(-l')\in {\rm Pic}(\tX)$, and even an $H$--covering, which replaces the universal abelian covering.
Once the covering is fixed, the analogues of the
 natural line bundles are defined via an eigensheaf decomposition as in
  (\ref{eq:01}) (and the line after it), and all
 the analytic filtrations $\{\cF(l')\}_{l'}$ can be defined as in
 (\ref{eq:03}),  and all the basic statements of that subsection can be reproved.
 Again, here in the case of cusps, such a natural covering exists, it is called the
 `discriminant covering' of the cusp. The point is that any two splittings
 $H_1(M,\Z)\to H$ can be identified by an automorphism of $\pi_1(M)$ \cite{WallcuspsII}.
See also \cite{NWnew2} for the definition of this covering and several other properties of it. Summed up, in the special case of cusps, any splitting gives basically the same
covering, on which we can rely.


\bekezdes
As we already said, in the case of cusps,
the cohomology of line bundles are topological (in the sense of (\ref{eq:vanishing})),
hence $\cP(\bt)$
also should be topological. Let us rewrite its coefficients in terms of $\chi$.
Let us fix some $l'\in \calS'$ and
write $\overline{I}:=\calv\setminus I$. Then from the exact sequence
\begin{equation}\label{eq:ESbar}
0\to \calO_{\tX}(-l'-E)\to \calO_{\tX}(-l'-E_I)\to \calO_{E_{\overline{I}}}(-l'-E_I)\to 0,
\end{equation}
and from the vanishing $h^1(\calO_{\tX}(-l' -E))=0 $ (since $Z_K=E$ and
$l'\in \calS'$, hence Grauert--Riemenschneider vanishing holds), we get that $\pp(l')=\sum_I
(-1)^{|I|} h^0(  \calO_{E_{\overline{I}}}(-l'-E_I))$. Let $J:=J(l',I)$ be as in Lemma
\ref{DL}. Then $h^0(\tX, \calO_{\tX}(-l'-E_I))=h^0(\tX, \calO_{\tX}(-l'-E_{J(l',I}))$, hence again by (\ref{eq:ESbar}) we get that one also has
 \begin{equation*}
 \pp(l')=\sum_I (-1)^{|I|} h^0(  \calO_{E_{\overline{J(l',I)}}}(-l'-E_{J(l',I)}))=
 \sum_I (-1)^{|I|} h^0(  \calO_{E_{\overline{J}}}(-l'-E_{J})).\end{equation*}
We will separate the case $l'=0$.
Note that  $\pp(0)=1$ for {\it any} singularity. Here this can be seen as follows.
For $I=\emptyset $ one has
$\overline{J}=\calv$ hence  $h^0(\calO_{E_{\overline{J}}}(-E_J))=1$. But for $I\not=\emptyset$,
$h^0(\calO_{E_{\overline{J}}}(-E_J))=0$ (since $J=\calv$). Hence $\pp(0)=1$.

If $l'\not=0$ then $h^1(\calO_{E_{\overline{J}}}(-l'-E_J))=0$. This for $I=\emptyset$
reads as $h^1(\calO_{E}(-l'))=0$. But by  Grauert--Riemenschneider vanishing
$h^1(\calO_{E}(-l'))=h^1(\calO_{\widetilde{X}}(-l'))$, which vanishes by (\ref{eq:vanishing}).

For $I\not=\emptyset$ it follows from the fact that each component of $\overline{J}$ is a string and all the Chern degrees are $\geq 0$. In particular,
$h^0(\calO_{E_{\overline{J}}}(-l'-E_J))=\chi(\calO_{E_{\overline{J}}}(-l'-E_J))=
\chi(l'+E)-\chi(l'+E_J)$. Hence
 \begin{equation}\label{eq:cusp*}\cP(\bt)=\bt^0+\sum_{l'\in \calS'\setminus \{0\}}
 \sum_I (-1)^{|I|+1} \chi(l'+E_{J(l',I)})\cdot \bt^{l'}.\end{equation}
 This can  be compared with the expression from
 (\ref{eq:cusp**}), valid in the rational homology sphere link case for certain minimally elliptic singularities (with rational homology sphere link).

 \subsection{The series $ \cZ(\bt)$ for cusps} \label{def:Zcusp}
 The topological linear subspace arrangement can be defined analogously
  as suggested by  (\ref{cor:tarrZ}). Namely, one sets
 $T(l'):=H^0(\calO_E(-l'))$ and $T_v(l'):= H^0(\calO_{E-E_v}(-l'-E_v))$ for each $v\in\calv$. Then one defines $$\cZ(\bt)=\sum_{l'\in \cS'}\, \chi_{{\rm top}}(\, \bP(T(l')\setminus \cup_{v}T_v(l'))\,)
\cdot \bt^{l'}.$$
Since we use the reduced structure of $E$, one can argue (similarly as in the
rational homology sphere case) that $\cZ(\bt)$ depends only on the combinatorics of the graph.
(This will follow from the next lemma and its proof as well.)

\begin{lemma} $\cZ(\bt)=\cP(\bt)$, or,
\begin{equation}\label{eq:Zcusp}
\cZ(\bt)=\bt^0+\sum_{l'\in \calS'\setminus \{0\}}
 \sum_I (-1)^{|I|+1} \chi(l'+E_{J(l',I)})\cdot \bt^{l'}.\end{equation}
\end{lemma}
\begin{proof}
One can proceed in two different ways. The first version is that one notice that
$H^0(\calO_{\tX}(-l'))\to H^0(\calO_{E}(-l'))$ is onto (since $h^1(\calO_{\tX}(-l'-E))=0$).
Hence, analytic and the topological linear subspace arrangements are the same.
The analytic one can be connected with $\cP(\bt)$ similarly as in Corollary \ref{cor:PArr},
hence $\cP=\cZ$ follows.

In a different way, one can analyse the validity of Proposition \ref{prop:tarr} as well.
One sees that for $l'\in\calS'\setminus \{0\}$ {\it (a)-(b)-(c)} are valid. {\it (d)}
 should be modified as follows: $\dim T(l')=h^0(\calO_E(-l'))=\chi(\calO_E(-l'))=-(E,l')$.
 Moreover,   {\it (e)}
 remain valid  as well. In particular, as in \ref{cor:tarreu} one proves
(\ref{eq:Zcusp}).
\end{proof}

\subsection{The computation of $\cP(\bt)=\cZ(\bt)$}
Assume first that $|\calv|\geq 3$.

(a) Assume that $l'=E_v^*$ for some $v$. Then $J=\emptyset $ for $I=\emptyset$
and $J=\calv$ otherwise. Hence, by (\ref{eq:cusp*}) $\pp(l')=1$. Similarly, if
$l'=kE^*_v$ for some $v\in\calv$ and $k\geq 2$ then
 $J=\emptyset $ for $I=\emptyset$,
and $J=\calv$ if $v\in I$, and $J=\calv\setminus v$ if $I\not=\emptyset$ and $v\not\in I$.
Hence, again by  (\ref{eq:cusp*}) $\pp(l')=1$.

(b) Assume that $v$ and $v'$ are adjacent vertices,  and $l'=kE^*_v+k'E^*_{v'}$ with
$k, k'>0$. Let $w\not= v'$ adjacent vertex  with $v$.
 Let $\cP_v:=\{I\subset \calv, v\in I\}$. Then its elements can be put in pairs
$(I,I\cup w) $ with $w\not\in I$, and with $J(l', I)=J(l', I\cup w)$. Hence the contribution in (\ref{eq:cusp*}) corresponding to the sum over $\cP_v$ is zero.
The same is true, by similar argument,  for the subset $\cP_{v,v'}:=\{I\subset \calv, v\not\in I, v'\in I\}$. Hence we remain with subsets $I$ with $I\cap \{v,v'\}=\emptyset$.
For them $J=\emptyset$ if $I=\emptyset$ (with contribution $\chi(l')$), and $J=\calv\setminus \{v,v'\}$ (with contribution $\chi(l')+1$) else. Hence $\pp(l')=1$ again.

(c) We claim that $\pp(l')=0$ in all other cases. Write $l'$ as $\sum _{v\in S}
a_vE^*_v$ with all $a_v>0$. Assume that the $E^*$ support $S$ of $l'$ is not $\calv$.
Fix a maximal connected string $\Gamma_1$ in $\calv\setminus S$. Let
$v, v'\in S$ be the two adjacent vertices of $\Gamma_1$. Then, by excluding the
cases already discussed, we know that $v\not=v'$ and $\Gamma_1\cup\{v,v'\}\not=\calv$.
Then we compute $\pp(l')$ via (\ref{eq:cusp*}).
Similarly as in case (b), the contribution in the sum corresponding to the subset
$I$ with $I\cap\{v,v'\}\not=\emptyset$ is zero (choose the adjacent $w$ in $\Gamma_1$).
Hence,  in the sequel we consider  sets $I$ with $I\cap \{v,v'\}=\emptyset$.
 For each such $I$, note that $\tilde{J}:=
J\setminus \{v,v'\}$ has the property that $\chi(l'+E_J)=\chi(l'+E_{\tilde{J}})$
($\dag$).
Indeed, if at some step of the algorithm from Lemma \ref{DL}(\ref{DL2}) we have $I_m=\tilde{J}$, and by the algorithm we have to add $v$ (or $v'$), then $(E_v, l'+E_{I_m})=1$, hence ($\dag$) holds. Write $I$ as $I_1\cup I_2$ with $I_1\subset \Gamma_1$, while $I_2\cap \Gamma_1=\emptyset$. Then $\tilde{J}$ also decomposes as $J_1\cup J_2$
with similar properties. Note also that $(l'+E_{J_2},E_{J_1})=0$ Hence
$\sum _{I_1\cup I_2}(-1)^{|I_1|+|I_2|}\chi(l'+E_{\tilde{J}})=
 \sum _{I_1\cup I_2}(-1)^{|I_1|+|I_2|}(\chi(l'+E_{J_2})+\chi(E_{J_1}))=
  \sum _{I_2}(-1)^{|I_2|}\chi(l'+E_{J_2})\sum_{I_1}(-1)^{|I_1|}+
   \sum _{I_1}(-1)^{|I_1|}\chi(E_{J_1})\sum_{I_2}(-1)^{|I_2|}=0$.

  Finally assume that $S=\calv$. Similarly as in ($\dag$) above, we have
  $\chi(l'+E_J)=\chi(l'+E_I)$, which equals $\chi(l'+E)-\chi(\calO_{E_{\bar{I}}}(-l'-E_I))$.
  (This identity can be proved via $h^1(\calO_{E_{\bar{I}}}(-l'-E_I))=0$ as well.)
  Now the vanishing $\sum_I  (-1)^{|I|} \chi(\calO_{E_{\bar{I}}}(-l'-E_I))=
  \sum_I  (-1)^{|I|}( \chi(E_{\bar{I}})-(E_{\bar{I}}, l'+E_I))=
  0$  follows from  combinatorial arguments.

\vspace{2mm}

 Hence we proved that whenever $|\calv|\geq 3$ then the following holds.

\begin{theorem} \label{bek:formulacusp}
$$\cP(\bt)=\cZ(\bt)=1+\sum_{v\in\calv} \frac{\bt^{E^*_v}}{1-\bt^{E^*_v}}+\sum _{(u,v)\in\cale}
\frac{\bt^{E^*_v}}{1-\bt^{E^*_v}}\cdot \frac{\bt^{E^*_u}}{1-\bt^{E^*_u}}.$$
\end{theorem}\noindent
By direct verification, the same formula holds for $\calv=\{v,u\}$ as well. Namely
$$\cP(\bt)=\cZ(\bt)=1+ \frac{\bt^{E^*_v}}{1-\bt^{E^*_v}}+
\frac{\bt^{E^*_u}}{1-\bt^{E^*_u}}+ 2\cdot
\frac{\bt^{E^*_v}}{1-\bt^{E^*_v}}\cdot \frac{\bt^{E^*_u}}{1-\bt^{E^*_u}}.$$
If $\calv=\{v\}$ then
$$\cP(t)=\cZ(t)=1+\bt^{E^*}/(1-\bt^{E^*})^2.$$

\subsection{ The motivic series associated with $\{{\rm ECa}^{l'}(E)\}_{l'}$ (general $(X,o)$)} \label{bek:MOTIVE}

Let us assume that $\tX$ is a good resolution of a normal surface singularity $(X,o)$
such that each $E_v$ is rational, however, we allow cycles in the graph
(hence in this subsection $(X,o)$ is not necessarily a cusp).
For any $l'\in\calS' $ let ${\rm ECa}^{-l'}(E)$ be the space of effective Cartier divisors of $E$, cf. \ref{bek:Eca} or \cite{NNI}.
It is a quasiprojective   variety, cf. \cite{Groth62,Kl,Kleiman2013,NNI}, non--empty if and only if $l'\in\calS'$. (See also \ref{bek:Eca}.)
Let us define the generating function of their classes in the Grothendieck group
\begin{equation}\label{eq:GrothGen}
\cZ^{{\rm ECa}}(\bL,\bt):= \sum_{l'\in \calS'}[{\rm ECa }^{-l'}(E)]\cdot \bt^{l'}.
\end{equation}

\begin{theorem}\label{th:GrothNEW} With the above notations
$$\cZ^{{\rm ECa}}(\bL,\bt)=\frac{
\prod_{(u,v)\in\cE}\ (1-\bt^{E_u^*}-\bt^{E_v^*}+\bL \bt^{E_u^*+E_v^*})}
{\prod_{v\in\cV} \ (1-\bt^{E^*_v})(1-\bL\bt^{E^*_v})}.$$
\end{theorem}
\begin{proof}
Write $l'=\sum_va_vE^*_v$ with $a_v\in\Z_{\geq 0}$. Recall that effective Cartier divisors are local nonzero
sections of the sheaf $\calO_E$
 up to local invertible elements of $\calO_E^*$. At any intersection point $p\in E_v\cap E_u$,
 $(v,u)=e\in \cE$, with local coordinates $(x,y)\in U$,
 $\{x=0\}=E_v\cap U$, $\{y=0\}=E_u\cap U$, an effective  Cartier divisor supported at $p$ has the form
 $D_{k_v,k_u}=c_v^ex^{k^e_v}+c_u^ey^{k^e_u}$ (in $\C\{x,y\}/(xy)$, up to an invertible element),
 where $c^e_v, \,c^e_u\in\C^*$. It is nonempty if and only if $k_v^e\geq 1$
 and $k_u^e\geq 1$.
 ${\rm ECa}^{-l'}(E)$ has a natural stratification according to
 the support of the divisors.
 The sum of the degrees of the divisors with support on $E_v$ should be $a_v$, out of this, say $a^\circ _v$,
 is provided by those supported on $E^\circ _v$, and the others by divisors supported on intersection points of type
 $E_v\cap E_u$. Such a divisor contribute in the degree  with $k^e_u$. Hence
 $a_v=a_v^\circ +\sum_{(v,u)=e\in\cE} k^e_u$ ($\dag_v$).  Hence, with fixed $\{a_v\}_v$, we consider the stratification of ${\rm ECa}^{l'}(E)$ according to the system $S(l'):=\{ \{a^\circ_v\}_v, (k_v^e,k_u^e)_{(v,u)\in\cE}\}$
 satisfying ($\dag_v$) for any $v$ and $k_u^e, k_v^e\geq 1$ for any  such
 edge $e$ which has a contribution.
 For $s=\{ \{a^\circ_v\}_v, (k_v^e,k_u^e)_{(v,u)\in\cE_s}\}\in S(l')$
 (where $\cE_s$  is the set of edges which contributes in this  stratum)
 the  corresponding stratum,
 ${\rm ECa}^{l'}_s (E)$,  satisfies  $$[{\rm ECa}^{l'}_s (E)]=(\bL-1)^{|\cE_s|}\cdot \prod_v \,[S^{a_v^\circ}E^\circ _v].$$

 First we compute the class $[S^{a_v^\circ}E^\circ _v]$ for any $v$. Let $E^c_v$ be a set consisting of $\kappa_v-1$
 distinct point, and $x$ a formal variable. Then $\sum_{i\geq 0} x^i[S^iE^\circ _v]\cdot
 \sum_{i\geq 0} x^i[S^iE^c _v]=\sum_{i\geq 0} x^i[S^i\C]=\sum _{i\geq 0} x^i\bL^i$. Hence
 $$\sum_{i\geq 0} x^i[S^iE^\circ _v]= (1-x)^{\kappa_v-1} / (1-\bL x).$$
 Next, let $S$ be the set of all the systems when we vary $l'$, that is,
 $S=\cup_{l'\in \calS'}S(l')$. Then, in $S$, $\{a_v\}_v$ is not fixed anymore,
 and the integers $a_v^\circ\geq 0$, $k_v^e,k_u^e\geq 1$ run independently.
Therefore, with the usual substitution $x_v=\bt^{E^*_v}$,
\begin{equation*}\begin{split}
Z^{{\rm ECa}}(\bL,\bx)&=\sum _{l'}\sum _{s\in S(l') } \big( \,
(\bL-1)^{|\cE_s|} \cdot \prod_v [S^{a^\circ_v}E^\circ_v]\cdot \prod_v x_v^{a^\circ _v+\sum
_{(v,u)=e\in \cale_s}k_u^e}\,\big)\\ &=
\sum_{s\in S} \big( \,
 \prod_v x_v^{a^\circ_v} [S^{a^\circ_v}E^\circ_v]\cdot
 (\bL-1)^{|\cE_s|} \cdot
 \prod_v x_v^{\sum _{(v,u)=e\in \cale_s}k_u^e}\,\big)\\
&=\prod_v \frac{(1-x_v)^{\kappa_v-1}}{1-\bL x_v}\cdot \sum_{\cale'\subset \cale}\ \
\prod _{(v,u)=e\in \cale'} (\bL-1)\cdot
\frac{x_v}{1-x_v}\cdot \frac{x_u}{1-x_u}\\
&=\prod_v \frac{(1-x_v)^{\kappa_v-1}}{1-\bL x_v}\cdot\ \
\prod _{(v,u)=e\in \cale}\big(\  1+ (\bL-1)\cdot
\frac{x_v}{1-x_v}\cdot \frac{x_u}{1-x_u}\ \big)\\
&=\prod_v \frac{(1-x_v)^{\kappa_v-1}}{1-\bL x_v}\cdot\ \
\prod _{(v,u)=e\in \cale} \, \frac{1-x_v-x_u+\bL x_vx_u}{
(1-x_v) (1-x_u)},
\end{split}\end{equation*}
which is equivalent with the needed expression.
\end{proof}

\begin{example}\label{ex:recover}
If the graph is a tree then ${\rm ECa}^{-l'}(E)=\bP(T(l')\setminus \cup_v T_v(l'))$  and
$\cZ^{{\rm ECa}}(\bL,\bt)=\cZ(\bL,\bt)$, hence we recover Theorem \ref{th:Groth}.
\end{example}
\begin{remark}
Since $\cZ_K-E=\sum_v(\kappa_v-2)E^*_v$ we have $\bt^{Z_K-E}\cdot \cZ(\frac{1}{\bt})=\cZ(\bt)$.
Similarly,  the motivic expression satisfies the functional equation
$$ \bL^{c(\Gamma)-1}\cdot \bt^{Z_K-E}\cdot \cZ(\bL,\bt)\big|_{t_v\mapsto (\bL t_v)^{-1}}=
\cZ(\bL, \bt).$$
\end{remark}

\subsection{The motivic $\cP(\bL,\bt)$ and $\cZ(\bL,\bt)$ for a cusp singularity}

Let us return to a cusp singularity $(X,o)$. Recall that  the vanishing
$h^1(\calO_{\tX}(-l'-E))=0$ guarantees  that the analytic and topological
arrangements are the
same. In particular, $\cP(\bL,\bt)=\cZ(\bL,\bt)$.
On the other hand,
for a cusp
$\cZ^{{\rm ECa}}(\bL=1,\bt)=\prod_v(1-\bt^{E^*_v})^{\kappa_v-2}\equiv 1$. However, this
`1' is not the right substitute for $Z(\bt)$!

Indeed, for any $l'\in \calS'$ we  have to consider the Abel map
$c^{-l'}:{\rm ECa}^{-l'}(E)\to {\rm Pic}^{-l'}(E)$, and for any/certain
$\cL\in {\rm Pic} ^{-l'}(E)$  the class
$[\bP(H^0(E,\cL)_{{\rm reg}})]= [ (c^{-l'})^{-1}(\cL)]$ instead of the whole
$[{\rm ECa}^{-l'}(E)]$. Let us denote this dependence
 by $\cL=\calO_E(-l')$   (though it can be
an arbitrary choice when $l'\not=0$ and $\calO_E(-l')=\calO_E$ for $l'=0$).
Hence, we define
$$\cZ(\bL,\bt):= \sum_{l'\in \calS'}\,
[\bP(H^0(E,\calO_E(-l'))_{{\rm reg}})]\cdot \bt^{l'}
=\sum_{l'\in \calS'}\,
[\bP(T(l')\setminus \cup_v T_v(l'))]\cdot \bt^{l'}.$$
The second identity follows similarly as in (\ref{eq:ECA}).

The substitution  $\bL=1$ replaces the element
$[\bP(H^0(E,\calO_E(-l'))_{{\rm reg}})]$ from the Grothendieck group
by $\chi_{top}( \bP(H^0(E,\calO_E(-l'))_{{\rm reg}}))$, hence
$\cZ(\bL=1,\bt)$ is exactly $\cZ(\bt)$ considered in \ref{def:Zcusp}.
\begin{theorem}\label{lem:cuspdiv}
$\cZ(\bL,0)=1$. Furthermore, for $l'\in\calS\setminus \{0\}$ one has
$$[\bP(H^0(E,\calO_E(-l'))_{{\rm reg}})]= [{\rm ECa}^{-l'}(E)] / (\bL-1).$$
\end{theorem}
\begin{proof}
First we claim that ${\rm Pic}^0(E)$, regarded as the kernel
of $c_1:H^1(\calO^*_E)\to L'$ in $H^1(\calO^*_E)$,  is the multiplicative group $\C^*$.
Indeed, the exponential map $exp(2\pi i \,\cdot):\calO_E\mapsto \calO_E^*$ induces
an isomorphism $H^1(\calO_E)/H^1(E,\Z)= \C/\Z\stackrel{exp(2\pi i\,\cdot )}{\longrightarrow}\C^*$.

Note that if $f:X\to Y$ is an algebraic locally trivial fibration with fiber $F$,
then $[X]=[Y][F]$. In the proof, instead of the algebraic locally triviality of the Abel map we will prove the locally triviality of its restrictions on the strata of ${\rm ECa}^{-l'}(E)$
consider in the proof of Theorem \ref{th:GrothNEW}.

First, consider the Abel map
$c^{-E^*_v}:{\rm ECa}^{-E^*_v}(E)\to {\rm Pic}^{-E^*_v}(E)$ for a fixed $v\in \calv$
(for more details see also \cite{NNI}).
The source consists of divisors of degree one supported on $E^\circ _v$, hence it can be identified with $E^\circ _v$. The Abel map can be described by Laufer integration. Recall that $p_g=1$ and $(X,o)$ is Gorenstein.
  Then we fix
the Gorenstein differential form in $\tX\setminus E$ with pole $E$, and we apply
the Laufer integration procedure \cite{Laufer72} \cite[p. 1281]{Laufer77}, or
 \cite[7.1]{NNI}.
We obtain that $c^{-E^*_v}:E^\circ _v\to \C^*$ is a  tautological bijection.

Next, consider for any $k\geq 1$ the Abel map
$c^{-kE^*_v}:{\rm ECa}^{-kE^*_v}(E)\to {\rm Pic}^{-kE^*_v}(E)$.
One has a  multiplicative structure $\prod _{i=1}^k{\rm ECa}^{-E^*_v}(E)\to
{\rm ECa}^{-kE^*_v}(E)$ given by union (sum) of divisors, while
$\prod _{i=1}^k{\rm Pic}^{-E^*_v}(E)\to
{\rm Pic}^{-kE^*_v}(E)$ given by the tensor product of  line bundles, that is,
multiplication in $\C^*$. These operators commute with the Abel maps.
Hence $c^{-kE^*_v}:S^kE^\circ _v=S^k\C^*\to \C^*$ is given by $c^{-kE^*_v}(\sum_{i=1}^kp_i)=\prod_i
c^{-E^*_v}(p_i)$ (product in $\C^*$). This is surjective and it is an
 algebraic fibration over $\C^*$.

Finally, consider a stratum corresponding to $e=(v,u)\in\cE$ consisting of divisors of
type  $D_{k_v,k_u}$. As a space, it  is isomorphic to $\C^*$. We claim that the Abel map
restricted to it is an isomorphism. Indeed, if we blow up $p=E_u\cap E_v$ several times
conveniently at
its infinitesimal close points, this strata can be identified with the divisors
on some newly created exceptional divisor $ E^\circ _{new}\simeq \C^*$ given by
$x^{d_v}+a=0$, where $a\in\C^*$ is the parameter of the stratum, and
$d_v={\rm gcd}\{ k_v,k_u\}$. By the discussion from the previous case,
applied for $S^{d_v}E^\circ _{new}\to \C^*$, we get that the subset given by
 $\{x^{d_v}+a=0\}_{a\in\C^*}$ maps isomorphically to $\C^*$.

Since the stratification is algebraic, and the Abel map is also algebraic, we are done.
\end{proof}
\begin{corollary}\label{cor:ZCuspMot}
For  a cusp singularity $\cZ(\bL,\bt)=\cP(\bL,\bt)$ is given by
$$\cZ(\bL,\bt)= 1+\frac{\cZ^{{\rm ECa}}(\bL,\bt)-1}{\bL-1},$$
where $\cZ^{{\rm ECa}}(\bL,\bt)$ is given in Theorem \ref{th:GrothNEW}.
\end{corollary}
\begin{remark}  We invite the reader to verify that
$$\Big(1+\frac{\cZ^{{\rm ECa}}(\bL,\bt)-1}{\bL-1}\Big)\Big|_{\bL=1}$$
provides exactly the expressions from \ref{bek:formulacusp} for
$\cZ(\bt)$ determined by a
direct verification.
(For this, one has to determine the first terms of the
Taylor expansion of $\bL\mapsto \cZ^{{\rm  ECA}}$ at $\bL=1$.)
In this way we provide a second independent alternative proof for the expression from
Theorem \ref{bek:formulacusp}.
\end{remark}
\begin{example}
For a cusp with one vertex one has
$\cZ(\bL,\bt)=1+\sum_{k\geq 1}\, [\bP^{k-1}]\cdot t^{kE^*}$.
\end{example}
\begin{remark} The expression from the right hand side of Theorem \ref{bek:formulacusp}
is rather surprising:  the series
$$z({\bf x})=1+\sum_{v\in\calv} \frac{x_v}{1-x_v}+\sum _{(u,v)\in\cale}
\frac{x_v}{1-x_v}\cdot \frac{x_u}{1-x_u}\in \Z[[{\bf x}]]$$
appears in a natural way in a rather different context in the literature. Indeed, identify  the
minimal good resolution graph $\Gamma$ with a simplicial complex with vertex set $\calv$ and having only 1--simplices, namely the edges. Then the Hilbert series of the graded Stanley--Reisner ring
associated with $\Gamma$ is exactly $z({\bf x})$ \cite{BH,MS}.
This might provide some new starting point  in the direction of the explicit determination of the
equation of the local ring of $(X,o)$ as well  (compare e.g. with \cite{StevensCusps}).
\end{remark}

\end{document}